\documentclass[12pt]{amsart}

\usepackage[utf8]{inputenc}
\usepackage[english]{babel}

\usepackage{mathrsfs}
\usepackage{amsthm}
\usepackage{amsmath}
\usepackage{amsfonts}
\usepackage{mathtools}
\mathtoolsset{showonlyrefs}
\usepackage{amssymb}

\usepackage{textcomp}
\usepackage{xcolor}

\usepackage{calrsfs}
\usepackage{csquotes}

\usepackage{tikz-cd}

\usepackage{comment}

\DeclareMathAlphabet{\mathcal}{OMS}{cmsy}{m}{n}

\theoremstyle{plain}
\newtheorem {theorem}{Theorem}[section]
\newtheorem {lemma}[theorem]{Lemma}
\newtheorem {corollary} [theorem]{Corollary}

\newtheorem {proposition} [theorem]{Proposition}
\theoremstyle{definition}
\newtheorem{definition}[theorem]{Definition}

\theoremstyle{remark}
\newtheorem{remark}[theorem]{Remark}

 \numberwithin{equation}{section} 
 
\newcommand{\R}{\mathbb{R}}

\newcommand{\N}{\mathbb{N}}

\newcommand{\D}{\mathcal{D}}
\renewcommand{\L}{\mathcal{L}}
\newcommand{\I}{\mathcal{I}}

\newcommand{\F}{\mathcal{F}}

\newcommand{\oo}{\infty}

\newcommand{ \eee}{\text{e}}

\newcommand{\x}{\times}

\newcommand{\e}{\varepsilon}

\renewcommand{\d}{\partial}
\renewcommand{\a}{\alpha}

\renewcommand{\k}{\kappa}
\newcommand{\g}{\gamma}
\newcommand{\la}{\lambda}
\newcommand{\de}{\delta}

\newcommand{\1}{{-1}}

\renewcommand{\gg}{\dot{\gamma}}

\newcommand{\vspan}{\mathrm{span}}

\newcommand{\Lie}{\mathrm{Lie}}

\renewcommand{\.}{\dots}

\newcommand{\mc}{\mathcal}
\newcommand{\mr}{\mathrm}

\newcommand{\wt}{\widetilde}
\newcommand{\IM}{\mathrm{Im}}

\newcommand{\Span}{\mathrm{span}}

\newcommand{\vect}{\mathrm{Vec}}

\newcommand{\be}{\begin{equation}}
\newcommand{\ee}{\end{equation}}

\renewcommand{\k}{\mathrm k}

\renewcommand{\k}{\kappa}

\newcommand{\Vecc}{{\mathrm{Vec}}}
\newcommand{\dxx}[1]{{\frac{\d{#1}}{\d x_1}}}
\newcommand{\dxy}[2]{{\frac{\d{#1}}{\d{#2}}}}

\begin{document}
	
\author[Le Donne]{Enrico Le Donne$^1$}
\address{$^1$University of Fribourg, Chemin du Mus\'ee~23, 1700 Fribourg, Switzerland}
\email{enrico.ledonne@unifr.ch}

\author[Paddeu]{Nicola Paddeu$^2$}
\address{$^2$University of Fribourg, Chemin du Mus\'ee~23, 1700 Fribourg, Switzerland}
\email{nicola.paddeu@unifr.ch}{}

\author[Socionovo]{Alessandro Socionovo$^3$}
\address{$^3$Laboratoire Jacques-Louis Lions, CNRS, Inria, Sorbonne Université, Université de Paris, France}
\email{alessandro.socionovo@sorbonne-universite.fr}

\title{Metabelian distributions and sub-Riemannian geodesics}

\maketitle

\begin{abstract}
	We begin by characterizing metabelian distributions in terms of principal bundle structures. Then, we prove that in sub-Riemannian manifolds with metabelian distributions of rank $r$, the projection of strictly singular trajectories to some $r$-dimensional manifold must remain within an analytic variety. As a consequence, for rank-2 metabelian distributions, geodesics are of class $C^1$.
	
	%After characterizing metabelian distributions in terms of principal bundle structures, we prove that in sub-Riemannian manifolds with metabelian distributions, the projection of strictly singular trajectories to some lower-dimensional manifold must remain within an analytic variety. As a consequence, for rank-2 metabelian distributions, geodesics are continuously differentiable. 
\end{abstract}

\tableofcontents

\section{Introduction}
The regularity of geodesics (i.e., isometric embeddings of intervals) is one of the most difficult and important unsolved problem in sub-Riemannian geometry. A sub-Riemannian manifold is a triplet $(M,\D,g)$, where $M$ is a smooth manifold, $\D\subset TM$ is a smooth bracket-generating subbundle of the tangent bundle $TM$, and $g$ is a smooth Riemannian metric on $\D$. A {\em horizontal curve} (also called an {\em admissible trajectory}) is an absolutely continuous curve that is almost everywhere tangent to $\D$. The distance between two points $p,q\in M$ is obtained by taking the infimum of lengths, with respect to the metric $g$, among all horizontal curves joining $p$ and $q$. All the basic needed facts and definitions concerning sub-Riemannian geometry are recalled in Section 2.

The Pontryagin Maximum Principle (from now on the PMP, see Proposition~\ref{prop:PMP} below) provides necessary conditions for horizontal curves in sub-Riemannian manifolds to be geodesics. Horizontal curves satisfying these necessary conditions are called {\em extremal} curves. The PMP derives from the first-order differental analysis of the {\em end-point map} (see~\cite[Section 8]{ABB20}). It shows that there exist two non-disjoint classes of extremals, that are called {\em normal} and {\em singular} (also called {\em abnormal}) curves, respectively. Normal geodesics, which are the only ones appearing in Riemannian geometry, are smooth. Abnormal curves correspond to singular points of the end-point map, i.e., points where the differential of the map is not surjective. 

A priori, abnormal curves are absolutely continuous by definition. So, the main difficulty in studying the regularity of sub-Riemannian geodesics is the presence of {\em strictly singular} (also called {\em strictly abnormal}) geodesics, i.e., geodesics that are abnormal but not normal. To study the regularity of strictly abnormal minimizers, a second-order analysis of the end-point map was developed: second-order necessary conditions (Goh conditions, see Proposition~\ref{prop:Goh} below) have to be satisfied by length-minimizing strictly abnormal curves.

The reader can find all the details concerning the end-point map, the PMP, and the Goh conditions in~\cite{AS96, ABB20, AgrSac}. Recently, these type of techniques were developed to a $n$-order analysis of the end-point map, providing general $n$-order Goh-type conditions for the minimality, see~\cite{BMS24}.

There exist examples of strictly abnormal geodesics: the first one was discovered in~\cite{Mon94} by Montgomery, and other classes of examples are studied in~\cite{LS95}. A recent algorithm to produce abnormal curves is presented in~\cite{Hak21}, but length-minimizing properties of these curves are not known. Anyway, all known examples of abnormal minimizers are smooth curves, and this shows how the nature of the problem is very difficult.

A new recent approach to the problem is based on the direct study of possible singularity points on the curve. If, at a point, the set of blowups (or, equivalently, the tangent cone, see~\cite[Definition 3.1]{MPV18-b}) is a singleton, then either the curve is differentiable at the point or the unique blowup is a corner, i.e., at the point, the right and the left derivatives of the curve exist and are linearly independent. Partially in~\cite{LM08}, and later in~\cite{HL16}, it was proved that geodesics cannot have this kind of singularities. Later, this result have been improved in~\cite{MPV18}, where it is shown that, at every point, the set of blowups to a geodesic contains at least one line (or a half line, for extreme points). In~\cite{HL23} an even stronger result is provided for Carnot groups (see Definition~\ref{def:Carnot_group}): the projections of blowups of geodesics to a Carnot group of one step lower are geodesics. Anyway, since geodesics cannot have corners, the uniqueness of the blowup for geodesics is equivalent to their differentiability, and it is still an open problem. In~\cite{MS21}, an attempt to solve this problem have been made by studying singularities of spiral type, where the set of blowups at the center contains all lines. 

The aim of this paper is to provide a new regularity result in a consistent class of sub-Riemannian manifolds, i.e., analytic {\em metabelian} rank-2 sub-Riemannian manifolds. Roughly speaking, a distribution is metabelian if, in a neighborhood of every point, there exists a frame $\F$ for the distribution such that the Lie algebra generated by iterated commutators of vector fields in $\F$ is abelian (see Definition~\ref{def:metabelian} below). Examples of metabelian distributions are Ehresmann connections on $G$-principal bundles, where the group $G$ is abelian (see Theorem~\ref{thm:Gpb} below). Such bundles naturally arise from the development of Gauge Theory in electrodynamics. This theory is known in literature as Yang-Mills Theory, see~\cite[Chapter 12]{Mon02}, and it is used to describe the motion of charged particles in electro-magnetic fields. As a matter of fact, the first example of a strictly abnormal minimizer cited above, was discoverd by Montgomery studying a particular $\R$-principal bundle on $\R^2$. Important classes of sub-Riemannian manifolds with metabelian distributions are metabelian Lie groups equipped with left-invariant distributions, such as the Hopf fibration, Jet spaces, and Lie groups of nilpotency step at most 3. We prove some results on metabelian distributions in Section 3. We point out that this section is of independent interest with respect to the regularity problem for strictly abnormal geodesics in sub-Riemannian manifolds.

Our main result is the following.

\begin{theorem}
	\label{thm:main}
	Geodesics in metabelian analytic equiregular sub-Riemannian manifolds of rank-2 are everywhere of class $C^1$, and analytic outside of a closed discrete set.
\end{theorem}

The proof of Theorem~\ref{thm:main} is given in Section 4. Actually, we will prove a slightly stronger statement: locally, every strictly abnormal minimizer is the lift of a curve whose image is inside a suitable analytic variety, see Theorem~\ref{thm:forte} below. The proof of this stronger statement relies upon a specific choice of coordinates, called exponential coordinates of second type (see Section 2.2 below). These coordinates are built starting from a frame $X_1,\.,X_r$ for the distribution. When the distribution is metabelian, in such system of coordinates, the differential of the flow of controls (of horizontal curves) is the identity on every iterated bracket of $X_1,\.,X_r$. This is our key technical Lemma~\ref{lem:differential_flow}. As a direct consequence of the latter result, the equation given by Goh conditions (in exponential coordinates of second type) provides the defining equation for the desired analytic variety. Once we proved that, locally, every strictly abnormal minimizer is the lift of a curve whose image is inside a suitable analytic variety, we apply this result to distributions of rank-2. A structure theorem for 1-dimensional analytic varieties of $\R^2$ combined with the non-minimality of corners for geodesics provides the result of Theorem~\ref{thm:main}.

We point out that some of our results do not hold when the metabelian assumption is dropped. Indeed, in Section 5, we provide an example of a spiral in $\R^2$ whose image is not contained in any 1-dimensional analytic variety of $\R^2$, and whose lift in a suitable non-metabelian sub-Riemannian manifold is a strictly abnormal curve. The possible minimality of such a curve is still an open problem.

We end this introductory part by spending some words on the regularity result we reached. In the first place, our result is an improvement (for metabelian manifolds) of~\cite{Sus14}, where it was stated that geodesics in analytic sub-Riemannian manifolds are analytic on an open dense set. Secondly, our Theorem~\ref{thm:main} is the third example of a $C^1$-regularity result for geodesics in sub-Riemannian manifolds, after~\cite{BCJPS20,BFPR22}. The classes of manifolds studied in these 3 papers have always rank-2, and have non-empty intersection. However, every of such classes has something peculiar with respect to the other two.

\medskip

{\bf Research funding.} N. Paddeu and E. Le Donne
were partially supported by 
the Swiss National Science Foundation
(grant 200021-204501 `\emph{Regularity of sub-Riemannian geodesics and applications}').

This project has received funding from the European Union’s Horizon 2020 research and innovation programme under the Marie Sk{\l}odowska-Curie grant agreement No 101034255. \includegraphics[width=0.65cm]{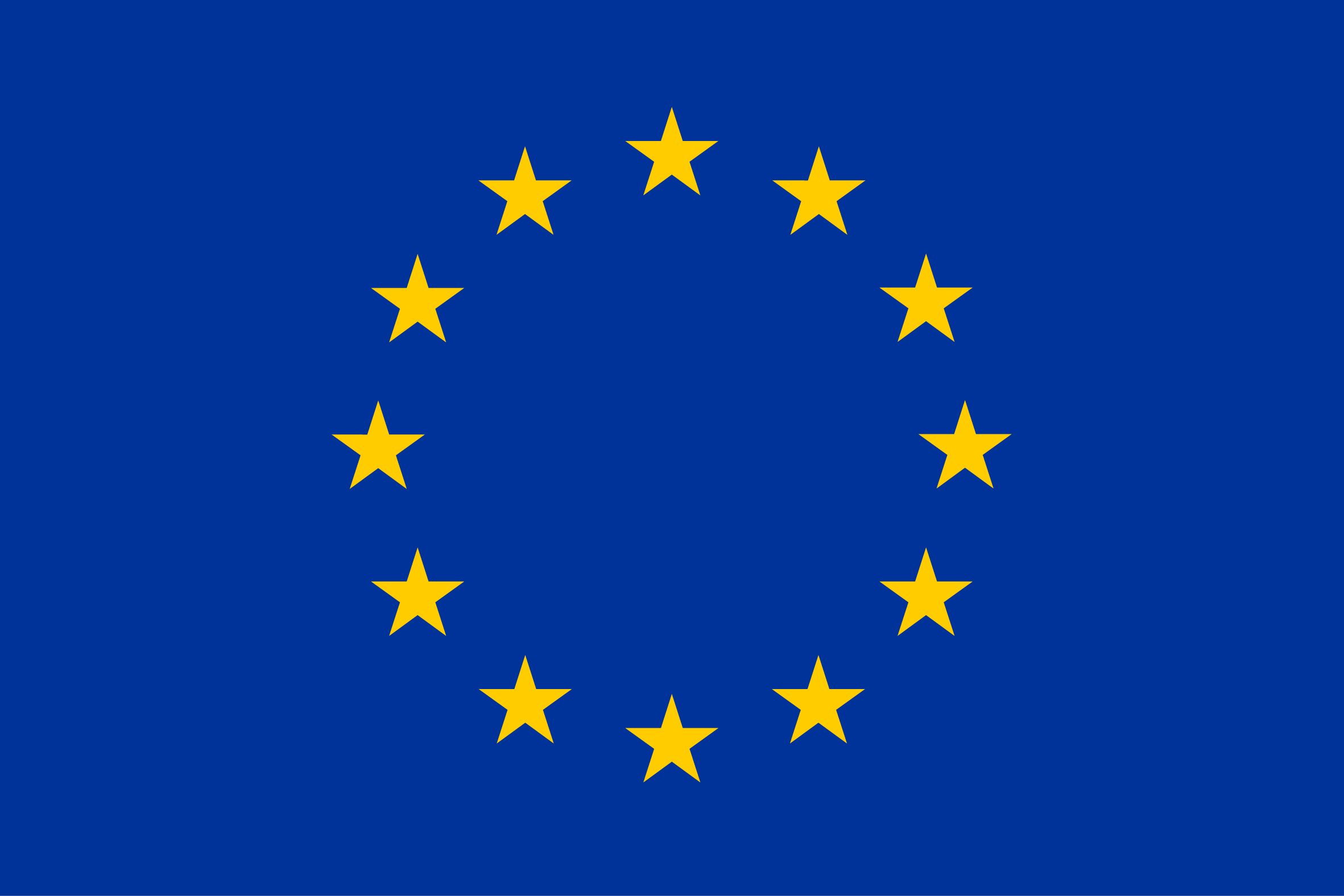}

\section{Preliminary notions}	

In this section we briefly present sub-Riemaniann manifolds. We recall the definitions of normal and abnormal trajectories and of blowup of a curve. Here and hereafter, $I=[a,b]\subset\R$ is an interval.

\begin{definition}
	\label{def:Carnot_group}
Let $M$ be a smooth (resp., analytic) manifold. A smooth (resp., analytic) {\em distribution} $\D$ on $M$ is a smooth (resp., analytic) subbundle of the tangent bundle $TM$. For all $p\in M$, we define by induction on $k\in\N$, $\D^1:=\D$, and
\begin{equation}
	\D^k_p:=\vspan\{[X_1,\.,X_j](p)\mid j\leq k, \, X_i\in\vect(M), \, X_i(q)\in\D_q \text{ for every } q\in M\}.
\end{equation}
A distribution $\D$ is {\em bracket-generating} if, for all $p\in M$, there exists $k$ such that $\D^k_p=T_pM$.

The {\em rank} at $p\in M$ of a distribution $\D$ is the integer $\dim(\D_p)$. We say that a distribution $\D$ has rank $r$ if $r=\dim(\D_p)$, for all $p\in M$. We say that a distribution $\D$ is {\em equiregular} if $\dim(\D^k_p)$ is constant in $p$, for all $k\in\N$.

A {\em sub-Riemannian manifold} is a connected manifold equipped with a bracket-generating distribution $\D$ and with a Riemannian metric $g$ defined on $\D$.
\end{definition} 

Throughout this paper, we consider sub-Riemannian manifolds $(M,\D,g)$ where $M,\D$ are analytic, and $\D$ is equiregular of rank $r$.

\begin{definition}
	Let $(M,\D,g)$ be a sub-Riemannian manifold. A {\em horiziontal curve} $\g:I\to M$ is an absolutely continuous curve such that $\Dot{\gamma}(t)\in \D_{\gamma(t)} \text{ for a.e. } t\in I$.
	
	The {\em Carnot-Carathéodory distance} on $M$ is defined, for all $x,y\in M$, by
	\begin{equation}
		\begin{split}
			d_{cc}(x,y):=\inf \bigg\{ \int_I\sqrt{g(\Dot{\gamma}(t),\Dot{\gamma}(t))}\mr d t \ \bigg|  \ \gamma:I\to M \text{ horizontal }, \gamma(0)=x, \gamma(1)=y\bigg\} .
		\end{split}
	\end{equation}
\end{definition}

The fact that $d_{cc}(x,y)<+\oo$ for every $x,y\in M$ is a direct consequence of the bracket-generating property of $\D$ and it is well known as Chow-Rashevskii Theorem (see~\cite{ABB20}).

\subsection{Length-minimizing curves in sub-Riemannian manifolds}

In this paper we are interested in the study of length-minimizing curves in sub-Riemannian manifolds. The first tool we will use is the Pontryagin Maximum Principle (PMP), which gives necessary conditions for curves to be geodesics. 

\begin{definition}
	Let $M$ be a sub-Riemannian manifold. Let $\gamma:I\to M$ be a horizontal curve. We say that $\g$ is {\em length-minimizing} (or a {\em length minimizer}) if, for every $t,s\in I$, $s<t$, we have $\l(\g|_{[s,t]})=d_{cc}(\g(t),\g(s))$. We say that $\g$ is a {\em geodesic} if, for every $t,s\in I$, we have $d_{cc}(\g(t),\g(s))=|t-s|$.
\end{definition}

\begin{remark}
	\label{rem:geodpbal}
	Notice that geodesics are parametrized by arc-length, and arc-length reparametrization of length minimizers are geodesics. We are only interested on the regularity of geodesics since smoothness can be destroyed by a bad reparametrization.
\end{remark}

\begin{definition}
	\label{def:flow_u}
	Let $M$ be a sub-Riemannian manifold with distribution $\D$. Fix $X_1,\.,X_k\in\mr{Vec}(M)$ such that $\D_p=\mr{span}\{X_1(p),\.,X_k(p)\}$ for all $p\in M$. We call \emph{controls} the elements of $L^2(I,\R^k)$. For every control $u$, we define the non-autonomous vector field
	\begin{equation}
		X_u(t,x):=\sum_{i=1}^{k}u_i(t)X_i(x).
	\end{equation} 	
	The {\em flow} of a control $u$ is the map $\Phi_u:U\subset\R\x M\to M$ solving the following ODE
	\begin{equation}
		\d_t\Phi_u(t,x)=X_u(t,\Phi(t,x)), \text{ for all } x\in M, \text{ and for a.e. } t\in U\cap (\R\times \{x\}),
	\end{equation}
	where $U\cap (\R\times \{x\})$ is connected and maximal. For a horizontal curve $\g:I\to M$ we say that $u$ is a {\em control associated to} $\g$ if 
	\begin{equation}
		\g(t)=\Phi_u(t,\g(0)).
	\end{equation}
\end{definition}

In the following, we denote the flow map putting the time $t$ as an apex, i.e., $\Phi_u^t(x):=\Phi_u(t,x)$. We also emphasize that if the vector fields $X_1,\.,X_k$ are a (local) frame for $\D$, then there is a unique control associated to each horizontal curve.

\begin{proposition}[PMP, {\cite[Theorem 3.59]{ABB20}}] 
\label{prop:PMP}
Let $(M,\D,g)$ be a sub-Riemannian manifold. Fix $X_1,...,X_k\in \Vecc (M)$ such that $\D_p=\Span(\{X_1(p),...,X_k(p)\})$ for every $p\in M$. Let $\gamma:I\to M$ be a geodesic, and let $u$ be a control of $\g$ with respect to $X_1,\.,X_k$. Then there exists a covector $\lambda\in T^*_{\gamma(0)}M$ such that it holds either
\begin{equation}
	\label{eq:PMP_normal}
	g(X,\gg(t))=\la((\Phi_u^t)^*X), \text{ for all } X\in\D_{\g(t)}, \text{ and for all } t\in I,
\end{equation}
or $\la\neq0$ and
\begin{equation}
	\label{eq:PMP_abnormal}
	\la\big((\Phi_u^t)^*X\big)=0, \text{ for all } X\in\D_{\g(t)}, \text{ and for all } t\in I.
\end{equation}
\end{proposition}

Horizontal curves satisfying at least one between~\eqref{eq:PMP_normal} and~\eqref{eq:PMP_abnormal} are called \emph{extremals}. The extremals satysfing~\eqref{eq:PMP_normal} are called \emph{normal}, the ones for which~\eqref{eq:PMP_abnormal} holds are called \emph{abnormal}. Abnormal curves for which~\eqref{eq:PMP_normal} does not hold for any $\la\neq0$ are called \emph{strictly abnormal extremals}. It is known that normal extremals are locally length-minimizing and smooth (and analytic if $M$, $\D$, and $g$ are analytic). Thus, if a non-smooth length-minimizing curve exists, it must be a strictly abnormal extremal.

\begin{proposition}[Goh conditions, {\cite[Theorem 12.13]{ABB20}}]
	\label{prop:Goh}
	In the setting of Proposition~\ref{prop:PMP}, if $\g$ is a strictly abnormal geodesic, we have
	\begin{equation}
		\label{eq:propGoh}
		\la\big((\Phi_u^t)^*Y\big)=0, \text{ for all } Y\in\D^2_{\g(t)}, \text{ and for all } t\in I.
	\end{equation}
\end{proposition}

\begin{remark}
	\label{rem:pmp_goh}
	We stress that Equations~\eqref{eq:PMP_abnormal} and~\eqref{eq:propGoh} do not depend on the metric $g$. Moreover, if $\D$ has rank-2, then every non-constant abnormal curve satisfies Goh conditions~\eqref{eq:propGoh} (see~\cite[Lemma 12.28]{ABB20}).
\end{remark}

\subsection{Exponential coordinates of the second kind} In this section we introduce in $M$ exponential coordinates of the second type centered at a point $p \in M$.

\begin{definition}
	\label{def:stratbas}
	Let  $M$ be a smooth manifold and let $\D\subset TM$ be a smooth equiregular distribution of rank $r$. Let $p\in M$ and let $X_1,\.,X_n\in\Vecc(U)$ be a local frame of $TM$ in a neighborhood $U$ of $p$. We say that $X_1,\.,X_n$ is a stratified basis at $p$ if 
	\begin{itemize}
		\item[i)] $X_1,\.,X_r$ is a local frame for $\D$ in $U$;
		\item[ii)] for every $i>r$, the vector field $X_i$ is an iterated commutator of $X_1,\.,X_r$.
	\end{itemize}
	If $X_1,\.,X_n$ is a stratified basis of $M$ at $p$, we define the weight of $X_i$, $i=1,\.,n$, setting
	\begin{equation}
		w_i:=\min\{k\in\N\,|\,\exists i_1,\.,i_k\in\{1,\.,r\}, X_i=[X_{i_1},\.,X_{i_k}]\}.
	\end{equation}
\end{definition}

\begin{definition}
	\label{def:exp_coord_2}
	Let  $M$ be a smooth manifold and let $\D\subset TM$ be a smooth equiregular distribution of rank $r$. Le $p\in M$, and let $X_1,\.,X_n$ be a stratified basis in a neighborhood $U$ of $p$. For a point $q\in U$, we call {\em exponential coordinates of the second type} associated to $X_1,\.,X_n$ the unique $n$-tuple $x=(x_1,\.,x_n)\in\R^n$ such that
	\begin{equation} \label{23}
		q = \Phi^{x_1}_{ X_1} \circ\.\circ\Phi^{x_n}_{X_n} (p),
	\end{equation}
	Above $\Phi_X^t$ denotes the flow of the vector field $X$ for time $t$.
\end{definition}

Since our problem has a local nature, we assume, without loss of generality, that $M=U=\R^n$ and $p=0$, using the identification given by the exponential coordinates of the second type. We also assume that the vector fields of the stratified basis that we consider are complete.

The following theorem is proved in~\cite{hermes} in the case of general rank. For us it will be important only the statement for rank-2 equiregular distributions, which is explicitly proved in~\cite{MS21}.

\begin{theorem}
	\label{giallo}
	Let  $M$ be an analytic manifold, with $0\in M=\R^n$, and let $\D\subset TM$ be an analytic bracket-generating equiregular distribution of rank $r$. Let $X_1,\.,X_n$ be a stratified basis at $0$. 
	In exponential coordinates of the second type (see Definition~\ref{def:exp_coord_2}) there exist analytic functions $A_{k,j}:\R^n\to\R$, $k=2,\.,r$, $j=r+1,\.,n$, such that the  vector fields $X_1,\.,X_r$ have the form
	\begin{equation} 
		\label{eq:campihermes}
		\begin{split}
			&X_1(x)=\partial_{x_1}, \\
			&X_k(x)=\partial_{x_k}+\sum_{j=r+1}^n{A_{k,j}(x)\partial_{x_j}},
		\end{split}
	\end{equation}
	for all $2\leq k\leq r$, and $x\in \R^n$.
\end{theorem}

\section{Metabelian distributions}

In this section, we are going to use the notions of $G$-principal bundles and Ehresmann connections. We briefly recall that if $G$ is a Lie group and $\pi:M\to B$ is a $G$-principal bundle, then a distribution $\D\subset TM$ is an Ehresmann connection if
\begin{align}
	\tag{E1}\label{eq:E1} &\text{$\D$ is invariant under the action of $G$;}\\
	\tag{E2}\label{eq:E2} &\text{for every $p\in M$, we have $\D_p\oplus \ker d_p\pi=T_pM$}.
\end{align}
For more details concerning these topics, we refer the reader to~\cite[Chapter 11]{Mon02}.

We denote by $\I_m^k:=\{(i_1,\.,i_k)|1\leq i_j\leq m, \, j=1,\.,k\}$ the set of multi-indices of length $k$ and variables less or equal than $m$.

\begin{definition}
	Let $M$ be a smooth manifold and let $X_1,\.,X_m\in\Vecc(M)$. For $k\in\N$ and for a multi-index $J=(j_1,\dots,j_k)\in\I_m^k$, we define the {\em iterated bracket} associated to $J$ as
	\begin{equation}
		X_J:=[X_{j_1},[ \dots,[  X_{j_{k-1} }, X_{j_k}]\cdots ]].
	\end{equation} 
	%We also define the {\em length} of $X_J$ as the length of the multi-index $J$, i.e., 
	%\begin{equation}
	%	\mathrm{len}(X_J)= \mathrm{len}(J) = k.
	%\end{equation}
\end{definition}

\begin{definition}
	\label{def:metabelian}
	Let $M$ be a smooth manifold with a smooth distribution $\D\subset TM$. We say that the distribution $\D$ is {\em metabelian} if for all $p\in M$ there exists a local frame $X_1,\.,X_r$ for $\D$ in a neighborhood $U$ of $p$ such that
	\begin{equation}
		\label{eq:metabelian}
		[X_I,X_J]=0,
	\end{equation}
	for every $I\in \I_r^{k_I}$ and $J\in \I_r^{k_J}$ with $k_I,k_J\geq2$. 
\end{definition}

The following result is the main theorem of this section: it shows that metabelian distributions naturally arise as Ehresemann connections on a specific class of principal bundles.

\begin{theorem}
	\label{thm:Gpb}
	Let $M$ be a $G$-principal bundle and let $\D$ be an Ehresemann connection. If $G$ is abelian, then $\D$ is metabelian.
\end{theorem}

\begin{proof}
	By assumption, $M$ is a $G$-principal bundle $\pi:M\to B$, with $G$ an abelian Lie group, and $\D$ is an Ehresmann connection.
	
	Fix $p\in M$. By definition of $G$-principal bundle, there exists a neighborhood $U$ of $p$ diffeomorphic to $\R^r\times G$, where $r=\dim(B)$. Moreover, by condition~\eqref{eq:E2}, the restriction $d_q\pi|_{\D_q}:\D_q\to T_{\pi(q)}B$ of $d_q\pi:\D_q\oplus \ker d_q\pi\to T_{\pi(q)}B$ is a linear isomorphism for every $q\in U$. We can identify $T_{\pi(q)}B$ with $\R^r$ and, after this identification, we can choose a local frame $X_1,\.,X_r$ for $\D$ in $U$ such that
	\begin{equation}
		\label{eq:choice_Xi}
		d_q\pi(X_i(q))=\d_{x_i}.
	\end{equation}
	
	With a slight abuse of notation, we write points $q\in U$ as couples $q=(x,g)$, with $x\in\R^r$ and $g\in G$. Moreover, we identify $T_qM$ with $\R^r\oplus T_gG$, and we write vectors $v\in T_qM$ as $v=v_1+v_2$, with $v_1\in\R^r$ and $v_2\in T_gG$. Then, as a consequence of~\eqref{eq:choice_Xi}, we have that the vector fields $X_i$, for all $i=1,\.,r$, are of the form
	\begin{equation}
		\label{eq:Xi}
		X_i(x,g)= \d_{x_i}+Y_i(x,g), \text{ for all $(x,g)\in \R^r\times G$},
	\end{equation}
	where $Y_i(x,g)$ is some vector in $T_gG$. 	We claim that formula~\eqref{eq:metabelian} holds for the vector fields in~\eqref{eq:Xi}.
		
	Fix $(x,g)\in U$. By~\eqref{eq:E1}, for all $i=1,\.,r$, the vector $Y_i$ appearing in~\eqref{eq:Xi} is of the form 
	\begin{equation}
		\label{eq:boh}
		Y_i(x,g)=dL_g \bar Y_i(x),
	\end{equation}
	for some $\bar Y_i:\R^r\to \Lie(G)\simeq T_e G$. Denoting by $\Phi^t$ the flow of $\d_{x_i}$ for time $t$ and by $\eee_i$ the $i$-th vector of the canonical basis of $\R^r$, we compute for all $i,j=1,\.,r$ the Lie derivative
	\begin{equation}
		\label{eq:boh2}
		\begin{split}
			\L_{\d_{x_i}}Y_j(x,g)&=\frac{\d}{\d\e}(\Phi^\e_{\d_{x_i}})^*Y_j(x,g)\Big|_{\e=0}\\
			%&=\frac{\d}{\d\e}Y_j(x+\e\eee_i,g)\Big|_{\e=0}\\
			&\stackrel{\eqref{eq:boh}}{=}dL_g\Big(\frac{\d}{\d\e}\bar Y_j(x+\e\eee_i)\Big|_{\e=0}\Big)\\
			&=dL_g\Big(\frac{\d\bar Y_j}{\d x_i}(x)\Big).
		\end{split}
	\end{equation} 
	Moreover, we also have for all $i,j=1,\.,r$
	\begin{equation}
		\label{eq:boh3}
		\begin{split}
			[X_i,X_j]=\L_{X_i}X_j
			&\stackrel{\eqref{eq:Xi}}{=}\L_{\d_{x_i}+ Y_i}(\d_{x_j}+Y_j)\\
			&\stackrel{\eqref{eq:boh}}{=} \L_{\d_{x_i}}Y_j + \L_{Y_i}\d_{x_j} + \L_{Y_i}Y_j\\
			&=\L_{\d_{x_i}}Y_j - \L_{\d_{x_j}}Y_i + \L_{Y_i}Y_j.
		\end{split}
	\end{equation}
	By~\eqref{eq:boh},~\eqref{eq:boh2}, and~\eqref{eq:boh3} we deduce for all $i,j=1,\.,r$
	\begin{equation}
		\begin{split}
			[X_i,X_j](x,g)&=dL_g\left(\dxy{\bar Y_j}{x_i}(x)-\dxy{\bar Y_i}{x_j}(x) + [\bar Y_i(x),\bar Y_j(x)]_{\Lie(G)}\right)\\
			&=dL_g\left(\dxy{\bar Y_j}{x_i}(x)-\dxy{\bar Y_i}{x_j}(x)\right),
		\end{split}
	\end{equation}
	where $[\cdot,\cdot]_{\Lie(G)}$ is the Lie bracket in $\Lie(G)$ and in the last identity we used that $G$ is abelian. Therefore we have for all $i,j=1,\.,r$
	\begin{equation}
		[X_i,X_j](x,g)=dL_g (\bar Y_{ij}(x)),
	\end{equation}
	for some $\bar Y_{ij}:\R^r\to\Lie(G)$. With the same strategy, one can prove that for every function $\bar Y:\R^r\to\Lie(G)$ and for all $k=1,\.,r$, if $Y(x,g):=dL_g\bar Y(x)$, it holds
	\begin{equation}
		[X_k,Y](x,g)=dL_g\left(\dxy{\bar Y}{x_k}(x) + [\bar Y_k(x),\bar Y(x)]_{\Lie(G)}\right)=dL_g\left(\dxy{\bar Y}{x_k}(x)\right).
	\end{equation}
	Therefore, for every multi-index $I$ there exists $\bar Y_I:\R^r\to\Lie(G)$ such that $X_I(x,g)=dL_g(\bar Y_I(x))$. Consequently, we have for all multi-indices $I,J$ of length greater than or equal to 2
	\begin{equation}
		\label{eq:fin_metab}
		[X_I,X_J](x,g)=dL_g\big([\bar Y_I(x),\bar Y_J(x)]_{\Lie(G)}\big).
	\end{equation}
	Since $G$ is abelian we have $[\bar Y_I(x),\bar Y_J(x)]_{\Lie(G)}=0$, proving our claim.
\end{proof}

\subsection{The analytic bracket-generating case} In general, Ehresemann connections are not bracket-generating distributions. This happens if and only if the holonomy group of the connection on the $G$-principal bundle coincides with the group $G$. In this section, we focus on (analytic) bracket-generating distributions, which are the only ones of interest in sub-Riemannian geometry.

\begin{theorem}
	\label{giallo2}
	Let $M$ be an analytic manifold equipped with an analytic equiregular distribution $\D$ of rank $r$. Then $\D$ is metabelian if and only if there exists a frame $X_1,\.,X_r$ of $\D$ such that, in exponential coordinates of second type, the functions $A_{k,j}$ of Theorem~\ref{giallo} depend only on the variables $x_1,\.,x_r$, i.e., we have
	\begin{equation}
		\label{eq:Xi2}
		\begin{split}
			&X_1(x)=\partial_{x_1}, \\
			&X_k(x)=\partial_{x_k}+\sum_{j=r+1}^n{A_{k,j}(x_1,\.,x_r)\partial_{x_j}}, \quad \text{for all $2\leq k\leq r$, and $x\in \R^n$}.
		\end{split}
	\end{equation}
\end{theorem}

The proof is similar to the one given in~\cite[Theorem 2.2]{MS21} in the case of $r=2$. We omit the details.

In the following proposition, we prove a local version of the opposite implication in Theorem~\ref{thm:Gpb}. Namely, we show that equiregular metabelian distributions (of rank $r$) are locally induced by a $\R^{n-r}$-principal bundle structure. 

\begin{proposition}
	Let $M$ be an analytic manifold, and let $\D\subset TM$ be an equiregular, analytic, bracket-generating, and metabelian distribution of rank $r$. Then $M$ is locally a $\R^{n-r}$-principal bundle and $\D$ is an Ehresmann connection for this bundle structure. 
\end{proposition} 

\begin{proof}
	Let $p\in M$. Since $M,\D$ are analytic, by Theorem~\ref{giallo2} there exists a neighborhood $U$ of $p$ diffeomorphic to $\R^n$ and a frame $X_1,\.,X_r$ for $\D$ in $U$ of the form~\eqref{eq:Xi2}.	
	
	Denoting points of $\R^n$ as couples $x=(x',x'')\in\R^r\oplus\R^{n-r}$, we have that
	\begin{equation}
		X_i(x+(0,\tau))=X_i(x),
	\end{equation}
	for all $x\in\R^n$ and $\tau\in\R^{n-r}$. Therefore, every $X_i$ (and so the distribution) is invariant under the action by translation of the additive group $\{0\}\x\R^{n-r}\subset\R^n$ on $\R^n$. Moreover, by~\eqref{eq:Xi2} we have that $\R^n=\D_x\oplus\R^{n-r}$, proving that $U$ is a $\R^{n-r}$-principal bundle and $\D$ is an Ehresmann connection for this bundle.
\end{proof}

We end this section with a useful formula for the differential of the flow for metabelian and bracket-generating distributions.

\begin{lemma}
	\label{lem:differential_flow}
	Let $X_1,\.,X_r\in\Vecc(\R^n)$ be vector fields of the form~\eqref{eq:Xi2}. Let $u\in L^2(I,\R^r)$. Denote by $\Phi:\R\x\R^n\to\R^n$ the flow of $u$ (see Definition~\ref{def:flow_u}). Then we have
	\begin{equation}
		\label{eq:differential_flow}
		(\Phi^t)_*\d_{x_i}=\d_{x_i}, \quad \text{for all } i\in\{r+1,\.,n\}.
	\end{equation}
\end{lemma}

\begin{proof}
	Let $I=[a,b]$. For all $1\leq j\leq n$, by~\eqref{eq:campihermes} we have
	\begin{align}
		&(\Phi^t)_j(x)=x_j+\int_a^t u_j(s)ds, \quad 1\leq j\leq r,\\
		&(\Phi^t)_j(x)=x_j+\int_a^t \sum_{k=1}^{r}u_k(s)A_{k,j}((\Phi^s)_1(x),\.,(\Phi^s)_r(x))ds \quad r+1\leq j\leq n.
	\end{align}
	Therefore we deduce
	\begin{equation}
		\frac{\d(\Phi^t)_j}{\d x_i}(x)=
		\begin{cases}
			1, \quad j=i,\\
			0, \quad j\neq i,
		\end{cases}
	\end{equation}
	for all $r+1\leq i\leq n$, $1\leq j\leq n$, and $x\in\R^n$. The latter equation implies~\eqref{eq:differential_flow}.
\end{proof}

\section{Proof of Theorem~\ref{thm:main}: $C^1$-regularity of geodesics}

In this section, $(M,\D,g)$ is a sub-Riemannian manifold. We are going to assume $M=\R^n$, $\D\subset TM$ is analytic, metabelian of rank $r$, and with frame $X_1,\.,X_r$ of the form~\eqref{eq:Xi2}.

We explicitly compute the bracket of $X_h,X_k$, that is
\begin{equation}
	\label{eq:X12}
	Y^{h,k}(x):=[X_h,X_k](x)\stackrel{\eqref{eq:Xi2}}{=}\sum_{j=r+1}^{n}\Big(\dxy{A_{k,j}}{x_h}(x_1,\.x_r)-\dxy{A_{h,j}}{x_k}(x_1,\.x_r)\Big)\d_{x_j},
\end{equation} 
where by convention we let $A_{1,j}=0$ for all $j=r+1,\.,n$.
We also define, for every covector $\la\in(\R^n)^*$, the analytic functions $F_\la^{h,k}:\R^r\to\R$,
\begin{equation}
	\label{eq:Fla}
	F_\la^{h,k}(x):=\la\left(\sum_{j=r+1}^{n}\Big(\dxy{A_{k,j}}{x_h}(x_1,\.x_r)-\dxy{A_{h,j}}{x_k}(x_1,\.x_r)\Big)\d_{x_j}\right), \; x\in\R^r.
\end{equation}

\begin{definition}
	\label{def:kg}
	For a horizontal curve $\g:I\to \R^n$, $\g=(\g_1,\.,\g_n)$, we denote by $\k_\g:I\to\R^r$ the linear projection of $\g$ on the first $r$ coordinates, that is $\k_\g:=(\g_1,\.,\g_r)$.
\end{definition}

In order to prove our Theorem~\ref{thm:main}, we in fact prove the following slightly stronger statement.

\begin{theorem}
	\label{thm:forte}
	Let $\D$ be an analytic, metabelian distribution of rank $r$ on $\R^n$ with frame $X_1,\.,X_r$ of the form~\eqref{eq:Xi2}. Let $\g:I\to\R^n$ be a horizontal curve satisfying Goh conditions~\eqref{eq:propGoh} for some covector $\la\in(\R^n)*\setminus\{0\}$. Then the image of $\k_\g$ is contained in the analytic variety
	\begin{equation}
		V_\la:=\{x\in\R^r\,|\, F^{h,k}_\la(x)=0, \text{ for all } h,k=1,\.,r\}\subset\R^r.
	\end{equation}
\end{theorem}

\begin{proof}
	By \eqref{eq:propGoh}, for all $h,k=1,\.,r$, we have that
	\begin{equation}
		\label{eq:Goh}
		\la\big((\Phi^t)^*Y^{h,k}(0)\big)=0, \quad \text{for all } t\in I.
	\end{equation}
	Here we are going to apply our key result in Lemma~\ref{lem:differential_flow}. Then, by~\eqref{eq:differential_flow} and~\eqref{eq:X12}, Equation~\eqref{eq:Goh} above reads
	\begin{equation}
		\label{eq:Goh2}
		\la\big(Y^{h,k}(\g(t))\big)=0, \quad \text{for all } t\in I.
	\end{equation}
	In particular, by~\eqref{eq:X12},~\eqref{eq:Fla} and~\eqref{eq:Goh2} it follows that, for all $h,k$,
	\begin{equation}
		\label{eq:F_la(k_g)=0}
		F^{h,k}_\la(\k_\g(t))=0,
		%\left(\sum_{j=3}^{n}\dxx{A_j}(\k_\g(t))\d_{x_i}\right)=0, 
	\end{equation}
	for all $t\in I$, and this complete the proof.	
\end{proof}

Combining Remark~\ref{rem:pmp_goh} and Theorem~\ref{thm:forte} we obtain the following corollary for rank $2$ distributions.

\begin{corollary}
	\label{cor:abnrk2}
	Let $\D$ be an analytic, metabelian distribution of rank $2$. Let $\g:I\to\R^n$ be an abnormal curve. Then the image of $\k_\g$ is contained in the analytic variety
	\begin{equation}
		V_\la:=\{x\in\R^2\,|\, F^{1,2}_\la(x)=0\}\subset\R^2.
	\end{equation}
\end{corollary}

We end this section with the proof of our main Theorem~\ref{thm:main}.

\begin{proof}[Proof of Theorem~\ref{thm:main}]

By Theorem~\ref{giallo}, a horizontal curve $\g$ is uniquely determined by its projection $\k_\g$ and, moreover, $\g$ has the same regularity of $\k_\g$ because they have the same control (with respect to different sub-Riemannian structures). Therefore, we prove the regularity of projections of geodesics.

Let $\g:I\to\R^n$ be a strictly abnormal geodesic, and let us prove that its projection $\k_\g$ is $C^1$, and analytic outside of a discrete set of point. 
 
By Corollary~\ref{cor:abnrk2}, the image of $\k_\g$ is contained in the analytic variety
\begin{equation}
	V_\la:=\{x\in\R^2\,|\, F^{1,2}_\la(x)=0\}\subset\R^2.
\end{equation}
It is known that there exists a closed discret set $S\subset V_\la$ such that $V_\la\setminus S$ is a real one-dimensional analytic manifold, and every connected component of $V_\la\setminus S$ can be parameterized by arc-length analytically (see~\cite[Appendix~B]{BFPR22}). Therefore $\k_\g$ is analytic on the set $I\setminus \k_\g^\1(S)$. Notice that, since $\k_\g$ is parameterized by arc-length and $\IM(\k_\g)\subset V_\la$, the set $\k_\g^{-1}(S)\subset I$ is a discrete set.

We are left to prove the $C^1$-regularity of $\k_\g$ in points of $S\cap\mr{Im}(\k_\g)$. Fix $x\in S$ and let $\bar t\in I$ be such that $\k_\g(\bar t)=x$. By Newton-Puiseaux Theorem (see~\cite[Chapter~2]{Cut04}, \cite[Appendix~B]{BFPR22}, or~\cite{Pui50}) in a neighborhood of $x$, we have that $V_\la\setminus\{x\}$ is a disjoint union
\begin{equation}
	V_\la\setminus\{x\}=\bigsqcup_{m=1}^N \big(\mr{Im}(\eta_j)\setminus\{x\}\big),
\end{equation}
for some $N\in\N$, and for all $j=1,\.,N$, the curve $\eta_j:[0,\de)\to\R^2$ satisfies
\begin{itemize}
	\item[i)] $\eta_j(0)=x$;
	\item[ii)] $\eta_j(t)\neq x$ for all $t\in(0,\de)$;
	\item[iii)] $\eta_j$ is parameterized by arc-length;
	\item[iv)] the limit $\lim_{t\to0^+}\dot\eta_j(t)$ exists.
\end{itemize}

By (iii) and being $\k_\g$ parameterized by arc-length, there exist (up to shrink $\de$) $j_1,j_2\in\{1,\.,N\}$ such that $\k_\g(\bar t-s)=\eta_{j_1}(s)$, and $\k_\g(\bar t+s)=\eta_{j_1}(s)$ for $s\in[0,\de)$. Therefore, by (iv), the limits $\lim_{t\to\bar t^-}\k_g(t)$, $\lim_{t\to\bar t^+}\k_g(t)$ exist. Since $\g$ is a geodesic, the two limits must coincide (see~\cite{HL16}). This complete the proof of the $C^1$-regularity of $\g$.
\end{proof}

\section{A counter-example in the non-metabelian case} 

\newcommand{\FF}{\mathbb{F}}
In this short section we provide an example that shows how the metabelian assumption for the distribution $\D$ is an essential ingredient in our proof of Theorem~\ref{thm:main}.

Let $\FF=\FF_{2,7}$ be the free nilpotent Lie group of rank 2 and step 7. Fix two left-invariant vector fields $X_1^L,X_2^L$ that generate the Lie algebra Lie$(\FF)$ of $\FF$. Let $\D$ be the distribution with frame $X_1^L,X_2^L$, and complete $X_1^L,X_2^L$ to a stratified basis $X_1^L,X_2^L,\.,X_n^L$ for $\D$ (see Definition~\ref{def:stratbas}).

\begin{remark}
	The distribution $\D$ is not metabelian. Indeed, assume that there exists a frame $X_1,X_2$ of the distribution such that \eqref{eq:metabelian} holds. On the one hand, we have $\Lie(\FF) \subseteq \vspan \{X_J \mid J\in I_m^2, m\leq 7 \}$. On the other hand, the dimension of $\Lie(\FF)$  is strictly bigger than the dimension of all nilpotent metabelian Lie algebras of rank $2$ and step $7$. Consequently, we have $ \vspan \{X_J \mid J\in I_m^2, m\leq 7 \}<\dim(\Lie(\FF))$, and we reached a contradiction.
\end{remark}

\begin{remark}
	\label{rem:diag_comm}
	Since $X_1^L,X_2^L$ are left-invariant, all the $X_j^L$'s are left-invariant. By~\cite[Proposition 1.2.8]{Gre-Co}, exponential coordinates of second type $\varphi_n^L:\R^n\to\FF$ (see Definition~\ref{def:exp_coord_2}) are a global diffeomorphism. Moreover, also the map $\varphi_2^L:\R^2\to\FF/[\FF,\FF]$ defined by
	\begin{equation}
		\varphi_2^L(x_1,x_2):=[\FF,\FF]\exp(x_2X_2^L)\exp(x_1X_1^L), \quad \text{for all } x_1,x_2\in\R,
	\end{equation} 
	is a global diffeomorphism (see~\cite[Theorem 1.2.12]{Gre-Co}). Therefore, the following diagram commutes
	\begin{equation}
		\begin{tikzcd}
			\R^n \arrow[r, "\varphi_n^L"] \arrow[d, "\pi"] & \FF \arrow[d, "\bar\pi"] \\
			\R^2 \arrow[r, "\varphi_2^L"]                     & {\FF/[\FF,\FF],}         
		\end{tikzcd}
	\end{equation}
	where $\pi(x_1,\.,x_n):=(x_1,x_2)$ and $\bar\pi$ is the canonical projection to the abelianization of $\FF$.
\end{remark}

Here and hereafter, for a stratified basis $\mc X=\{X_1,\.,X_n\}$ for $\D$ we short denote $\varphi_n^{\mc X}:\R^n\to\FF$ the exponential coordinates of second type, see Definition \ref{def:exp_coord_2}.

Now we build our example. Let $\k:[0,1]\to\R^2$ be the spiral curve defined by
\begin{equation}
	\label{eq:spiral}
	\k(t):=\begin{cases}
		t\eee^{i\phi(t)}, \quad &t\in(0,1],\\
		0, \quad &t=0,
	\end{cases}
\end{equation}
where $\phi(t):=-\log t$ is the phase of the spiral. We set $\bar\k:=\varphi_2^L\circ\k$, and we define $\bar\g:[0,1]\to\FF$ to be the unique horizontal lift of $\bar\k$.

\begin{proposition}
	The curve $\bar\g$ is abnormal in $\FF$. Moreover, for every stratified basis $\mc X=\{X_1,\.,X_n\}$ for $\D$ in a neighborhood of $1_\FF$, the curve $\k_{\bar\g}:=\pi\circ(\phi_n^{\mc X})^\1\circ\bar\g$ is not contained in an analytic variety of dimension 1. 
\end{proposition}

\begin{proof}
	By \cite[Section 5.1]{H20} the curve $\bar\g$ is abnormal.
	
	We prove the second claim of the statement. Let $F:\R^2\to\R$ be an analytic function such that
	\begin{equation}
		\label{eq:F(k)=0}
		F(\k_{\bar\g}(t))=0, \quad \text{for all } t\in[0,1].
	\end{equation}
	Fix $x\in\R^2$. Since the map $\R^2\ni (x_1,x_2)\mapsto \pi\circ (\varphi_n^{\mc X})^\1 (\exp(x_2X_2^L)\exp(x_1X_1^L))$ is a local analytic diffeomorphism, by \eqref{eq:spiral} we have that there exists $t_j\searrow0$ such that $\k_{\bar\g}(t_j)\in\R x$. Therefore, by~\eqref{eq:F(k)=0}, the real analytic function $\R\ni s\mapsto F(sx)\in\R$ has a non-isolated zero at $s=0$. Thus $F(sx)=0$ for all $s\in\R$. Since $x\in\R^2$ was arbitrary, we in fact proved that $F$ must be identically 0. This prove that the smallest analytic variety containing the image of $\k_{\bar\g}$ is $\R^2$.
\end{proof}

This example shows that, if we drop the metabelian assumption, Corollary~\ref{cor:abnrk2} does not hold. Therefore, to prove the $C^1$ regularity of geodesics in non-metabelian sub-Riemannian manifolds a different approach is required. Notice that, the minimality of $\bar\g$ in $\FF$ is still an open problem.

\bibliographystyle{plain}

\bibliography{Biblio}

\begin{thebibliography}{10}

\bibitem{AS96}
A.~A. Agrachev and A.~V. Sarychev.
\newblock Abnormal sub-{R}iemannian geodesics: {M}orse index and rigidity.
\newblock {\em Ann. Inst. H. Poincar\'{e} Anal. Non Lin\'{e}aire},
  13(6):635--690, 1996.

\bibitem{ABB20}
Andrei Agrachev, Davide Barilari, and Ugo Boscain.
\newblock {\em A comprehensive introduction to sub-{R}iemannian geometry},
  volume 181 of {\em Cambridge Studies in Advanced Mathematics}.
\newblock Cambridge University Press, Cambridge, 2020.
\newblock With an appendix by Igor Zelenko.

\bibitem{AgrSac}
Andrei~A. Agrachev and Yuri~L. Sachkov.
\newblock {\em Control theory from the geometric viewpoint}, volume~87 of {\em
  Encyclopaedia of Mathematical Sciences}.
\newblock Springer-Verlag, Berlin, 2004.
\newblock Control Theory and Optimization, II.

\bibitem{BCJPS20}
D.~Barilari, Y.~Chitour, F.~Jean, D.~Prandi, and M.~Sigalotti.
\newblock On the regularity of abnormal minimizers for rank 2 sub-{R}iemannian
  structures.
\newblock {\em J. Math. Pures Appl. (9)}, 133:118--138, 2020.

\bibitem{BMS24}
Francesco Boarotto, Roberto Monti, and Alessandro Socionovo.
\newblock Higher order goh conditions for singular extremals of corank 1.
\newblock {\em Arch. Rational Mech. Anal.}, 248(23), 2024.

\bibitem{Gre-Co}
Laurence Corwin and Frederick~P Greenleaf.
\newblock {\em Representations of nilpotent {L}ie groups and their
  applications: {V}olume 1, {P}art 1, {B}asic theory and examples}, volume~18.
\newblock Cambridge university press, 1990.

\bibitem{Cut04}
Steven~Dale Cutkosky.
\newblock {\em Resolution of singularities}, volume~63 of {\em Graduate Studies
  in Mathematics}.
\newblock American Mathematical Society, Providence, RI, 2004.

\bibitem{BFPR22}
A.~Belotto da~Silva, A.~Figalli, A.~Parusi\'{n}ski, and L.~Rifford.
\newblock Strong {S}ard conjecture and regularity of singular minimizing
  geodesics for analytic sub-{R}iemannian structures in dimension 3.
\newblock {\em Invent. Math.}, 229(1):395--448, 2022.

\bibitem{Hak21}
Eero Hakavuori.
\newblock O{DE} trajectories as abnormal curves in {C}arnot groups.
\newblock {\em J. Differential Equations}, 300:458--486, 2021.

\bibitem{H20}
Eero Hakavuori.
\newblock O{DE} trajectories as abnormal curves in {C}arnot groups.
\newblock {\em J. Differential Equations}, 300:458--486, 2021.

\bibitem{HL16}
Eero Hakavuori and Enrico Le~Donne.
\newblock Non-minimality of corners in subriemannian geometry.
\newblock {\em Invent. Math.}, 206(3):693--704, 2016.

\bibitem{HL23}
Eero Hakavuori and Enrico Le~Donne.
\newblock Blowups and blowdowns of geodesics in {C}arnot groups.
\newblock {\em J. Differential Geom.}, 123(2):267--310, 2023.

\bibitem{hermes}
Henry Hermes.
\newblock Nilpotent and high-order approximations of vector field systems.
\newblock {\em SIAM Rev.}, 33(2):238--264, 1991.

\bibitem{LM08}
Gian~Paolo Leonardi and Roberto Monti.
\newblock End-point equations and regularity of sub-{R}iemannian geodesics.
\newblock {\em Geom. Funct. Anal.}, 18(2):552--582, 2008.

\bibitem{LS95}
Wensheng Liu and H\'{e}ctor~J. Sussman.
\newblock Shortest paths for sub-{R}iemannian metrics on rank-two
  distributions.
\newblock {\em Mem. Amer. Math. Soc.}, 118(564):x+104, 1995.

\bibitem{Mon94}
Richard Montgomery.
\newblock Abnormal minimizers.
\newblock {\em SIAM J. Control Optim.}, 32(6):1605--1620, 1994.

\bibitem{Mon02}
Richard Montgomery.
\newblock {\em A tour of subriemannian geometries, their geodesics and
  applications}, volume~91 of {\em Mathematical Surveys and Monographs}.
\newblock American Mathematical Society, Providence, RI, 2002.

\bibitem{MPV18}
Roberto Monti, Alessandro Pigati, and Davide Vittone.
\newblock Existence of tangent lines to {C}arnot-{C}arath\'{e}odory geodesics.
\newblock {\em Calc. Var. Partial Differential Equations}, 57(3):Paper No. 75,
  18, 2018.

\bibitem{MPV18-b}
Roberto Monti, Alessandro Pigati, and Davide Vittone.
\newblock On tangent cones to length minimizers in {C}arnot-{C}arath\'{e}odory
  spaces.
\newblock {\em SIAM J. Control Optim.}, 56(5):3351--3369, 2018.

\bibitem{MS21}
Roberto Monti and Alessandro Socionovo.
\newblock Non-minimality of spirals in sub-{R}iemannian manifolds.
\newblock {\em Calc. Var. Partial Differential Equations}, 60(6):Paper No. 218,
  20, 2021.

\bibitem{Pui50}
Victor Puiseux.
\newblock Recherches sur les fonctions alg{\'e}briques.
\newblock {\em Journal de math{\'e}matiques pures et appliqu{\'e}es},
  15:365--480, 1850.

\bibitem{Sus14}
Héctor~J. Sussmann.
\newblock A regularity theorem for minimizers of real-analytic subriemannian
  metrics.
\newblock In {\em 53rd IEEE Conference on Decision and Control}, pages
  4801--4806, 2014.

\end{thebibliography}

\end{document}